\documentclass[12pt]{amsart}

\usepackage{graphics, amsmath, amsfonts, amssymb, amsthm, amscd, mathrsfs}

\input{xy}
\xyoption{all}

\newtheorem{theorem}{Theorem}
\newtheorem{proposition}[theorem]{Proposition}
\newtheorem{lemma}[theorem]{Lemma}
\newtheorem{corollary}[theorem]{Corollary}

\newtheorem{introtheorem}{Theorem}

\theoremstyle{definition}

\newtheorem{example}[theorem]{Example}
\newtheorem{remark}[theorem]{Remark}

\newcommand{\A}{{\mathscr A}}
\newcommand{\C}{{\mathbb C}}
\newcommand{\N}{{\mathscr N}}
\renewcommand{\O}{{\mathscr O}}
\newcommand{\R}{{\mathbb R}}
\newcommand{\Z}{{\mathbb Z}}
\newcommand{\Aut}{{\operatorname{Aut}}}
\newcommand{\codim}{{\operatorname{codim}\,}}
\newcommand{\Hom}{{\operatorname{Hom}}}
\newcommand{\inv}{{^{-1}}}
\renewcommand{\sl}{{\operatorname{/\!\!/}}}

\title{An Oka principle for equivariant isomorphisms}

\author{Frank Kutzschebauch, Finnur L\'arusson, Gerald W.~Schwarz}

\address{Frank Kutzschebauch, Institute of Mathematics, University of Bern, Sidlerstrasse 5, CH-3012 Bern, Switzerland}
\email{frank.kutzschebauch@math.unibe.ch}

\address{Finnur L\'arusson, School of Mathematical Sciences, University of Adelaide, Adelaide SA 5005, Australia}
\email{finnur.larusson@adelaide.edu.au}

\address{Gerald W.~Schwarz, Department of Mathematics, Brandeis University, Waltham MA 02454-9110, USA}
\email{schwarz@brandeis.edu}



\date{20 March 2013.  Most recent minor changes 14 August 2013}

\thanks{F.~Kutzschebauch was partially supported by Schweizerischer Nationalfond grant 200021-140235/1.  F.~L\'arusson was partially supported by Australian Research Council grant DP120104110.  F.~ L\'arusson and G.~W.~Schwarz would like to thank the University of Bern for hospitality and financial support.}

\begin{document}

\begin{abstract}  
Let $G$ be a reductive complex Lie group acting holomorphically on normal Stein spaces $X$ and $Y$, which are locally $G$-biholomorphic over a common categorical quotient $Q$.  When is there a global $G$-biholomorphism $X\to Y$?

If the actions of $G$ on $X$ and $Y$ are what we, with justification, call generic, we prove that the obstruction to solving this local-to-global problem is topological and provide sufficient conditions for it to vanish.  Our main tool is the equivariant version of Grauert's Oka principle due to Heinzner and Kutzschebauch.  

We prove that $X$ and $Y$ are $G$-biholomorphic if $X$ is $K$-contractible, where $K$ is a maximal compact subgroup of $G$, or if $X$ and $Y$ are smooth and there is a $G$-diffeomorphism $\psi:X\to Y$ over $Q$, which is holomorphic when restricted to each fibre of the quotient map $X\to Q$.  We prove a similar theorem when $\psi$ is only a $G$-homeomorphism, but with an assumption about its action on $G$-finite functions.  When $G$ is abelian, we obtain stronger theorems.  Our results can be interpreted as instances of the Oka principle for sections of the sheaf of $G$-biholomorphisms from $X$ to $Y$ over $Q$.  This sheaf can be badly singular, even for a low-dimensional representation of $\mathrm{SL}_2(\C)$.  

Our work is in part motivated by the linearisation problem for actions on $\C^n$.  It follows from one of our main results that a holomorphic $G$-action on $\C^n$, which is locally $G$-biholomorphic over a common quotient to a generic linear action, is linearisable.
\end{abstract}

\maketitle
\tableofcontents

\section{Introduction}  \label{s:introduction}

\noindent
In renowned work from the late 1950s, Grauert showed that a holomorphic principal $G$-bundle over a Stein space, where $G$ is a complex Lie group, has a holomorphic section if it has a continuous section \cite{Grauert}.  In fact, every continuous section can be deformed to a holomorphic section.  This is an instance of the Oka principle, a metatheorem supported by many actual theorems, saying that on Stein spaces, there are only topological obstructions to solving holomorphic problems that can be cohomologically formulated.  In a seminal paper of 1989, Gromov showed that the structure group is immaterial, so Grauert's theorem holds for any holomorphic fibre bundle whose fibre is a complex Lie group \cite{Gromov}.  And recently, Forstneri\v c has shown that Grauert's theorem holds even more generally for sections of any holomorphic submersion over a Stein space with the structure of a stratified holomorphic fibre bundle with complex Lie groups as fibres \cite{Forstneric2010}.  (We should say that we have not stated the theorems of Grauert, Gromov, and Forstneri\v c in their full strength.)  For more information on the Oka principle, see the monograph \cite{Forstneric2011} and the survey \cite{Forstneric-Larusson}.

In this paper, we prove Oka principles for sections of sheaves of groups that arise naturally in geometric invariant theory and that can be so singular, even for a low-dimensional representation of $\mathrm{SL}_2(\C)$ (Example \ref{e:canonical-bad-example}), that they are not represented by a complex space over the base.  The motivation for our study is to extend the Oka principle to \lq\lq singular bundles\rq\rq\ that lie beyond the reach of the theorems mentioned above, and at the same time to explore applications of the Oka principle in geometric invariant theory.  Also, there is more specific motivation coming from the so-called linearisation problem, which we describe at the end of this section.

For more details on the following, see the next section.  Let $G$ be a reductive complex Lie group.  Let $X$ and $Y$ be normal Stein spaces (always taken to be connected) on which $G$ acts holomorphically.   The categorical quotients $X\sl G$ and $Y\sl G$ are normal Stein spaces.  Assume that there is a biholomorphism $\tau:X\sl G\to Y\sl G$ that locally lifts to $G$-equivariant biholomorphisms between $G$-saturated open subsets of $X$ and $Y$.  We use $\tau$ to identify the quotients, and call the common quotient $Q$ with quotient maps $p:X\to Q$ and $r:Y\to Q$.   Our assumption, then, is that there is an open cover $(U_i)_{i\in I}$ of $Q$ and $G$-equivariant biholomorphisms $\phi_i:p^{-1}(U_i)\to r^{-1}(U_i)$ over $U_i$ (meaning that $\phi_i$ descends to the identity map of $U_i$).  We express the assumption by saying that $X$ and $Y$ are \textit{locally $G$-biholomorphic over a common quotient}.  

We want to conclude that there is a \textit{global} $G$-biholomorphism $X\to Y$.  If $X$ and $Y$ are what we, with justification, call generic (see below), we prove that the obstruction to solving our local-to-global problem is topological and provide sufficient conditions for it to vanish.

If $X$ is smooth, the common quotient $Q$ has a natural stratification, the \textit{Luna stratification}.  We call the corresponding stratified space the \textit{Luna quotient}.  There is a unique open stratum $Q_\mathrm{pr}$, the \textit{principal stratum}, and we set $X_\mathrm{pr}=p^{-1}(Q_\mathrm{pr})$.  If $X$ is only normal, then we
still have a stratification by \textit{isotropy type}.  There is a unique open stratum $Q_\mathrm{pr}$ in this case also, and the two definitions coincide when $X$ is smooth.  We say that $X$ (or the $G$-action on $X$) is \textit{$2$-principal} if $X\setminus X_\mathrm{pr}$ has codimension at least 2 in $X$.  If, in addition, $X_\mathrm{pr}$ consists of closed orbits with trivial stabiliser, we say that $X$ (or the $G$-action on $X$) is \textit{generic}.  In this case, the quotient map $X_\mathrm{pr}\to Q_\mathrm{pr}$ is a principal $G$-bundle \cite[Cor.~5.5]{Snow}.  We justify the term \lq\lq generic\rq\rq\ in Remark \ref{r:justification} below.  Our assumptions on $X$ and $Y$ show that $X$ is generic if and only if $Y$ is generic.

\begin{remark}  \label{r:assumptions}
 If $X$ and $Y$ are smooth and  locally $G$-biholomorphic over a common quotient (in particular, if $X$ and $Y$ are $G$-biholomorphic), then $X$ and $Y$ have isomorphic Luna quotients.  In \cite[Ex.~4.4]{Schwarz2013b}, there is an example of an automorphism of a Luna quotient of a generic $G$-module which does not lift over any neighbourhood of the image of the origin.  Thus an isomorphism of Luna quotients need not lift to a $G$-biholomorphism, even locally.  By slice theory, if the Luna quotients are isomorphic, then there are $G$-biholomorphisms $\phi_i:p^{-1}(U_i)\to r^{-1}(U_i)$ as above, except that $\phi_i$ need not descend to the identity map of $U_i$.  \textit{We do not know whether $X$ and $Y$ having isomorphic Luna quotients implies that they are locally $G$-biholomorphic over a common quotient.}
\end{remark}

Let $\psi_{ij}$ be the $G$-biholomorphism defined as $\phi_i^{-1}\circ\phi_j$ on $p^{-1}(U_i\cap U_j)$.  Then $(\psi_{ij})$ is a cocycle with respect to the open cover $(U_i)$ of $Q$ with coefficients in the sheaf $\A$ of groups of $G$-biholomorphisms of $X$ over $Q$.  There is a $G$-biholomorphism $X\to Y$ over $Q$ if and only if the cocycle splits, so the obstruction to $X$ and $Y$ being $G$-biholomorphic over $Q$ is an element of the cohomology set $H^1(Q, \A)$.

The important consequence of genericity is that the cocycle $(\psi_{ij})$ may be viewed as consisting of $G$-equivariant holomorphic maps $p^{-1}(U_i\cap U_j)\to G$, where $G$ acts on the target $G$ by conjugation (Lemma \ref{l:g-biholomorphisms}).  The cocycle thus defines a holomorphic principal bundle $E$ over $X$ with two commuting actions of $G$, one being part of the principal bundle structure, and the other making the projection $E\to X$ equivariant.  The bundle $E$ with the two $G$-actions is holomorphically trivial if and only if the cocycle splits.

We are now able to apply our fundamental tool, the equivariant version of Grauert's Oka principle due to Heinzner and Kutzschebauch \cite{Heinzner-Kutzschebauch}, which implies that $E$ is holomorphically trivial if it is topologically trivial.  We conclude that in the generic case, the obstruction to $X$ and $Y$ being $G$-biholomorphic is topological (Theorem \ref{t:topological-obstruction}).

Using, among other things, the equivariant version of the theory of universal bundles over classifying spaces, we go on to prove the first main result of the paper (Theorems \ref{t:torus} and \ref{t:K-contractible}).  Let $K$ be a maximal compact subgroup of $G$.  We say that $X$ is $K$-contractible if the identity map of $X$ can be joined to a constant map by a continuous path of $K$-equivariant continuous maps $X\to X$.  The value of the constant map is then a fixed point of the $K$-action on $X$, and hence of the $G$-action as well.
 
\begin{introtheorem}  \label{t:first-main}
Let $G$ be a reductive complex Lie group, and let $K$ be a maximal compact subgroup of $G$.  Let $X$ and $Y$ be normal Stein spaces on which $G$ acts holomorphically and generically, such that $X$ and $Y$ are locally $G$-biholomorphic over a common quotient.  If $X$ is $K$-contractible, then $X$ and $Y$ are $G$-biholomorphic.  If $G$ is abelian and $X$ is smooth, it suffices that $X$ be $\mathbb Z$-acyclic.
\end{introtheorem}

Our second main result is the following (Corollary \ref{c:strict-diffeomorphisms}).

\begin{introtheorem}  \label{t:second-main}
Let $G$ be a reductive complex Lie group acting holomorphically and generically on Stein manifolds $X$ and $Y$, which are locally $G$-biholomorphic over a common quotient $Q$.  Suppose there is a $G$-equivariant diffeomorphism $X\to Y$ over $Q$, which is holomorphic when restricted to each fibre of the quotient map.  Then $X$ and $Y$ are $G$-biholomorphic over $Q$.
\end{introtheorem}

In future work we hope to remove the genericity assumption, but this will require different methods.

We would like to interpret Theorems \ref{t:first-main} and \ref{t:second-main} as follows.  Let $\mathscr I$ be the sheaf of sets of $G$-biholomorphisms from $X$ to $Y$ over $Q$.  It is locally isomorphic to the sheaf $\A$, and is an $\A$-torsor, meaning that each stalk of $\A$ acts freely and transitively on the corresponding stalk of $\mathscr I$.  Theorems \ref{t:first-main} and \ref{t:second-main} each provide a sufficient condition for $\mathscr I$ to have a holomorphic section.  Theorem \ref{t:first-main} says that it does if $X$ is $K$-contractible.  Theorem \ref{t:second-main} says that it does if it has a smooth section.

We conjecture that $\mathscr I$ has a holomorphic section if it has a continuous section, but this we cannot prove in general.  We offer two partial results in this direction.  The first (Theorem \ref{t:codim-3}) says that if  %
$X$ is smooth and $G$ is abelian and the complement of the principal stratum $Q_\mathrm{pr}$ has codimension at least 3 in $Q$, then $\mathscr I$ has a holomorphic section over $Q$ if it has a continuous section merely over $Q_\mathrm{pr}$. 

The second result (Theorem \ref{t:strongly}) replaces the diffeomorphism in Theorem \ref{t:second-main} by a homeomorphism satisfying a technical condition described at the beginning of Section \ref{s:continuous}.  Thus, if $\mathscr I$ has a continuous section satisfying this condition, then $\mathscr I$ has a holomorphic section.

We remark that we do not prove Theorem \ref{t:second-main} and Theorem \ref{t:strongly} by deforming the diffeomorphism or homeomorphism in question to a biholomorphism.  Rather, we deform it to a diffeomorphism or homeomorphism of a kind that we call \textit{special}, the existence of which implies the existence of a biholomorphism by the equivariant Oka principle of Heinzner and Kutzschebauch.

We conclude this section with a few words about the linearisation problem.  The groups of holomorphic and algebraic automorphisms of $\C^n$ are infinite-dimensional and quite mysterious for $n\geq 2$ in the holomorphic case and $n\geq 3$ in the algebraic case.  It is of interest to study complex Lie subgroups of these groups, up to conjugacy.  The problem of linearising actions of reductive groups on $\C^n$ has attracted much attention both in the algebraic and holomorphic settings (\cite{Huckleberry}, \cite{ Kraft1996}). 

\smallskip\noindent
\textbf{Linearisation problem.}  Let a reductive complex Lie group $G$ act algebraically (resp.\ holomorphically) on $\C^n$.  Is it conjugate inside the group of algebraic (resp.\ holomorphic) automorphisms of $\C^n$ to a subgroup of $\mathrm{GL}_n(\C)$?  In other words, is there an algebraic (resp.\ holomorphic) change of coordinates on $\C^n$ that makes $G$ act by linear transformations?
\smallskip

The first counterexamples for the algebraic linearisation problem were constructed by Schwarz \cite{Schwarz1989} for $n\geq 4$.  His examples are holomorphically linearisable.  The holomorphic linearisation problem for $\C^*$-actions on $\C^n$, $n\geq 4$, was solved in the negative by Derksen and Kutzschebauch \cite{Derksen-Kutzschebauch}, who constructed holomorphic $\C^*$-actions on $\C^n$ whose stratified categorical quotients are not isomorphic to the stratified quotient of any linear action.  For $n\geq 5$, there are even holomorphically parametrised families of mutually holomorphically inequivalent holomorphic $\C^*$-actions on $\C^n$ \cite{Kutzschebauch-Lodin}.

Our contribution to the linearisation problem is the following consequence of Theorem \ref{t:first-main} (Corollaries \ref{c:first-corollary} and \ref{c:Franks-dream}).

\begin{introtheorem}
Let $X$ be a Stein manifold on which a reductive complex Lie group $G$ acts holomorphically.  If $X$ is locally $G$-biholomorphic over a common quotient to a generic $G$-module $V$, then $X$ is $G$-biholomorphic to $V$.
\end{introtheorem}

Note that the theorem gives a sufficient condition for a Stein manifold to be biholomorphic to $\C^n$.

\smallskip\noindent
\textit{Acknowledgement.}  We thank G.~Tomassini for pointing us to the reference \cite{Banica-Stanasila}, which helped us complete the proof of Theorem \ref{t:codim-3}.  We also thank the referees for valuable comments that helped us improve the exposition.


\section{Generic actions}  \label{s:generic}

We start with some background.  For more information, see \cite{Luna} and \cite[Sec.~6]{Snow}.  Let $X$ be a normal Stein space with a holomorphic action of a reductive complex Lie group $G$.  The categorical quotient $Q=X\sl G$ of $X$ by the action of $G$ is the set of closed orbits in $X$ with a reduced Stein structure that makes the quotient map $p:X\to Q$ the universal $G$-invariant holomorphic map from $X$ to a Stein space.  Since $X$ is normal, $Q$ is normal.  If $U$ is an open subset of $Q$, then $\O_X(p^{-1}(U))^G \cong \O_Q(U)$. 
If $X$ is a $G$-module, then $Q$ is just the complex space corresponding to the affine algebraic variety with coordinate ring  $\O_\mathrm{alg}(X)^G$.  We say that a subset of $X$ is \textit{$G$-saturated\/} if it is a union of fibres of $p$.  If $Gx$ is a closed orbit, then the stabiliser (or isotropy group) $G_x$ is reductive.  We say that closed orbits $Gx$ and $G{x'}$ have the same \textit{isotropy type} if $G_x$ is $G$-conjugate to $G_{x'}$. Thus we get the \textit{isotropy type  stratification} of $Q$ with strata whose labels are  conjugacy classes of reductive subgroups of $G$.

Let $H$ be a reductive subgroup of $G$ and let $Z$ be a normal Stein $H$-space. Define $G\times^H Z$ to be the orbit space of the free $H$-action on $G \times Z$ given by $h\cdot (g,z) = (g h^{-1} , h z)$.  We denote the $H$-orbit of $(g,z)$ by $[g,z]$. Since $G\times Z$ is Stein and normal, so is $G\times^H Z$. Note that $G$ acts naturally on $G\times^HZ$ on the left.

Let $q\in Q$ and take a point $x$ in the unique closed $G$-orbit in $p^{-1}(q)$.  Set $H=G_x$.  The \textit{slice theorem} states that there is a locally closed $H$-saturated Stein subvariety $S$ of $X$ (the \textit{slice}) containing $x$ such that $G\times^H S\to X$, $[g,s]\mapsto gs$, is a $G$-biholomorphism onto a $G$-saturated neighbourhood of $x$ in $X$.  If $x$ is a smooth point of $X$, the $H$-module  $T_x X / T_x(Gx)$ is called the \textit{slice representation} at $x$.  In this case, we may assume that the slice $S$ is $H$-biholomorphic to an $H$-saturated neighbourhood of the origin $0$ in $T_x X / T_x(Gx)$, with $0$ corresponding to $x$.  Since $H$ is reductive, we can identify $T_x X / T_x(Gx)$ with an $H$-stable complement $W$ to $T_x(Gx)$ in $T_x X$.  Write the $H$-module $W$ as $W^H \oplus W'$.  We may choose the slice $S$ to be $H$-biholomorphic to $B_1\times B_2$, where $0\in B_1\subset W^H$ and $0\in B_2\subset W'$.  Then $B_1$ maps biholomorphically onto a neighbourhood $U$ of $p(x)$ in the stratum through $p(x)$.  Thus $x$ has a $G$-saturated neighbourhood biholomorphic to $U\times(G\times^H B_2)$, and $p:p^{-1}(U)\to U$ is a trivial bundle with fibre $G\times^H \N(W')$, where $\N(W')$ is the null cone of $W'$, that is, the union of the $H$-orbits whose closure contains~$0$.

Assume that $X$ is smooth and that $G_x=gG_{x'}g^{-1}=G_{gx'}$, where $Gx$ and $Gx'$ are closed.  We say that $Gx$ and $Gx'$ have the same \textit{slice type} if the slice representations of $G_x$ at $x$ and $gx'$ are isomorphic.   Declaring points in $Q$ to be equivalent if the closed orbits above them have the same slice type defines a holomorphic stratification of $Q$, called the \textit{Luna stratification}.  The strata are smooth locally closed subvarieties of $Q$.  We call $Q$, viewed as a stratified complex space with each stratum labelled by the isomorphism class of the corresponding slice representation, the \textit{Luna quotient} of $X$ by the action of $G$.  An isomorphism of Luna quotients is a biholomorphism respecting the additional structure.  Precisely one Luna stratum is open in $Q$: this is the \textit{principal stratum} $Q_\mathrm{pr}$.  The corresponding stabiliser, well defined up to conjugation in $G$, is called the \textit{principal stabiliser}.  The closed orbits above the principal stratum are called \textit{principal orbits}.  The stratification of $Q$ by isotropy type  is coarser than the Luna stratification, but the connected components of the strata in the two stratifications are the same.  Both stratifications are locally finite, meaning that each point of $Q$ has a neighbourhood that intersects only finitely many strata.  The isotropy type stratification exists even if $X$ is not smooth, and then there is still an open and dense principal stratum $Q_\mathrm{pr}$.

Let $X$ and $Y$ be normal Stein spaces (always assumed connected) on which a reductive complex Lie group $G$ acts holomorphically, such that $X$ and $Y$ are locally $G$-biholomorphic over a common quotient $Q$.  Call the quotient maps $p:X\to Q$ and $r:Y\to Q$.  As in the introduction, there is an open cover $(U_i)$ of $Q$ and a cocycle $(\psi_{ij})$ with coefficients in the sheaf $\A$ of groups of $G$-biholomorphisms of $X$ over $Q$.  The obstruction to $X$ and $Y$ being $G$-biholomorphic is the corresponding element in $H^1(Q,\A)$.

\begin{example}  \label{e:obstruction}
Let $Q$ be a Stein manifold with $H^2(Q,\Z)\neq 0$, and let $L$ and $M$ be non-isomorphic holomorphic line bundles on $Q$.  Let $X$ and $Y$ be the manifolds obtained from $L$ and $M$, respectively, by removing the zero sections.  Vector bundles over Stein manifolds are Stein, and the complement of a hypersurface in a Stein manifold is Stein, so $X$ and $Y$ are Stein.  The actions of $\C^*$ on $L$ and $M$ by scalar multiplication in each fibre restrict to actions on $X$ and $Y$.  The actions are obviously free and generic.  The categorical quotient of both $X$ and $Y$ is $Q$ with a trivial stratification.  Clearly, $X$ and $Y$ are locally $G$-biholomorphic over $Q$.  It is easily seen that a $\C^*$-equivariant biholomorphism $X\to Y$ over $Q$ would extend to an isomorphism $L\to M$.  Conversely, an isomorphism $L\to M$ restricts to a $\C^*$-equivariant biholomorphism $X\to Y$ over $Q$.  Here, the sheaf $\A$ of $G$-biholomorphisms of $X$ over $Q$ is simply the sheaf $\O^*$ of nowhere-vanishing holomorphic functions.  Indeed, the obstruction to $X$ and $Y$ being $G$-biholomorphic over $Q$ is given by the element of $H^1(Q,\A)\cong H^2(Q,\Z)$ represented by $L\otimes M^*$.
\end{example}

So as not to break the flow of this section, we have postponed to the next section an example showing that the sheaf $\A$ can be badly behaved in that it need not be represented by a complex space.

In the introduction, we defined what it means for the action of $G$ on a normal Stein space $X$ to be generic and 2-principal.  The action of $G$ on $X$ is said to be \textit{stable} if there is a nonempty open subset of $X$ consisting of closed orbits; equivalently, $X_\mathrm{pr}=p^{-1}(Q_\mathrm{pr})$ consists of closed orbits.  We can reduce the stable 2-principal case to the generic case as follows.

\begin{proposition}
Let $X$ be stable and 2-principal and let $H$ be a principal stabiliser.  Let $X_H$ be the union of the irreducible components of the $H$-fixed point set $X^H$ which intersect $X_\mathrm{pr}$.  Then there is a $G$-equivariant biholomorphism
\[ \phi:G\times^{N_G(H)} X_H \to X,\quad [g,x]\mapsto gx. \]
Moreover, $X_H$ is normal and the action of $N_G(H)/H$ on $X_H$ is generic.

\end{proposition}

Here, $N_G(H)$ denotes the normaliser of $H$ in $G$.  Without the assumption that $X$ is stable and 2-principal, $X\sl G$ is biholomorphic to $X_H\sl N_G(H)$ \cite{Luna-Richardson}.  See also \cite[Thm.~7.5]{Schwarz1995}.

\begin{proof}
Clearly $X_H$ is $N_G(H)$-stable.  Set $X'=G\times^{N_G(H)} X_H$.    Suppose that $x\in X_H\cap X_\mathrm{pr}$ and $gx \in X_H$.  Then the stabiliser $G_{gx}$ contains $H$ and is conjugate to $H$, so $G_{gx}=H$.  But $G_{gx}=gG_xg^{-1}=gHg^{-1}$, so $g\in N_G(H)$.  Hence every principal orbit intersects $X_H$ in an $N_G(H)$-orbit with stabiliser $H$, and 
\[ Gx= G\times^{N_G(H)}(Gx\cap X_H).\]  
By construction then, $\phi$ induces a biholomorphism $\phi_\mathrm{pr}: X'_\mathrm{pr} \to X_\mathrm{pr}$, and $X'_\mathrm{pr}$ is open and dense in $X'$ since $X_\mathrm{pr}\cap X_H$ is dense in $X_H$.  Since $\codim X\setminus X_\mathrm{pr}\geq 2$, the inverse of $\phi_\mathrm{pr}$ extends to be holomorphic on $X$.  Hence $\phi$ is a biholomorphism.  Observe that the Luna quotient of $X_H$ by $N_G(H)/H$ is determined by that of $X$ by $G$, and vice versa.  Since $X'$ is a bundle over $G/N_G(H)$, if $\codim X_H\setminus (X_H)_\mathrm{pr}<2$, then $\codim X'\setminus (X')_\mathrm{pr}<2$: a contradiction.  Hence $\codim X_H\setminus (X_H)_\mathrm{pr} \geq 2$.  Finally, $X_H$ is normal since $X'\cong X$ is normal.
\end{proof}

\begin{corollary}
Let $X$ and $Y$ be locally $G$-biholomorphic over a common quotient and let $H$ be a principal stabiliser.  Suppose that $X$ is stable and 2-principal.  Then $Y$ is stable and 2-principal.  Let $X'$, $Y'$, $X_H$, and $Y_H$ be as above.  Then $X$ is $G$-biholomorphic to $Y$ if and only if $X'$ is $G$-biholomorphic to $Y'$ if and only if $X_H$ is $N_G(H)/H$-biholomorphic to $Y_H$.
\end{corollary}

Using the results above, one can prove versions of our main theorems with the hypothesis of genericity replaced by the assumption that the actions are 2-principal and stable.  We leave this as an exercise for the reader.

\begin{remark}  \label{r:justification}
Our use of the term \textit{generic} is rigorously justified in the case of $G$-modules.  If $G$ is simple, then, up to isomorphism, all but finitely many $G$-modules $V$ with $V^G=0$ are 2-principal and stable \cite[Cor.~11.6 (1)]{Schwarz1995}.  There is a similar result for semisimple groups \cite[Cor.~11.6 (2)]{Schwarz1995}.  If a $G$-module $V$ is 2-principal and stable, then the principal stabiliser $H$ is the kernel of the action of $G$ on $V$, so by replacing $G$ by $G/H$, we may assume that the principal stabiliser is trivial \cite[Rem.~2.6]{Schwarz2013}.  A \lq\lq random\rq\rq\ $\C^*$-module is generic, although infinite families of counterexamples exist.  More precisely, a faithful $n$-dimensional $\C^*$-representation without zero weights is generic if and only if it has at least two positive weights and at least two negative weights and any $n-1$ weights are coprime.  As justification for our use of the term \textit{generic} for arbitrary Stein $G$-manifolds $X$, we note that $X$ is generic if and only if each of its slice representations is generic.
\end{remark}

For us, the important consequence of $X$ being generic is that sections of $\A$ over an open subset $U$ of $Q$ may be identified with $G$-equivariant holomorphic maps $p^{-1}(U)\to G$, where $G$ acts on the target $G$ by conjugation: $h\cdot g=hgh^{-1}$.

\begin{lemma}  \label{l:g-biholomorphisms}
Let $X$ be a normal Stein space with a generic holomorphic action of a reductive complex Lie group $G$.  Let $p:X\to Q$ be the categorical quotient map.  Let $U$ be an open subset of $Q$.  There is a bijective correspondence between $G$-equivariant biholomorphisms $\psi$ of $p^{-1}(U)$ over $U$ and $G$-equivariant holomorphic maps $\gamma:p^{-1}(U)\to G$, where $G$ acts on itself by conjugation, given by $\psi(x)=\gamma(x)x$.
\end{lemma}

\begin{proof}
Clearly, every $G$-equivariant holomorphic map $\gamma:p^{-1}(U)\to G$ induces a $G$-biholomorphism of $p^{-1}(U)$ over $U$ by the formula $x\mapsto \gamma(x) x$.  We need to show that every $G$-biholomorphism is induced by a unique such map.

The assumption that the action is generic implies that holomorphic functions, and therefore also holomorphic maps into Stein spaces, extend uniquely from $X_\mathrm{pr}$ to $X$.  Hence we may assume that $Q_\mathrm{pr}=Q$, so $p:X\to Q$ is a principal $G$-bundle.

Now if $\psi:p^{-1}(U)\to p^{-1}(U)$ is a $G$-biholomorphism over $U$, then, since the action on each fibre is free and transitive, there is a unique map $\gamma:p^{-1}(U)\to G$ with $\psi(x)=\gamma(x)x$ for all $x\in p^{-1}(U)$, and, since the action is free,
\[ (\gamma(gx)g)x=\psi(gx)= g\psi(x)=g\gamma(x)x \]
implies that $\gamma$ is $G$-equivariant.  To show that $\gamma$ is holomorphic, we holomorphically and $G$-equivariantly trivialise $p$ over a small open subset $V$ of $U$, making $p^{-1}(V)$ isomorphic to $G\times V$ such that $G$ acts by left multiplication in the first component.  Then $\gamma(g,y)=(\mathrm{pr}_1\circ\psi(g,y))g^{-1}$, so $\gamma$ is holomorphic.
\end{proof}

We may now view the cocycle $(\psi_{ij})$ as consisting of $G$-equivariant holomorphic maps $p^{-1}(U_{ij})\to G$.  Ignoring the $G$-equivariance for the moment, the cocycle defines a holomorphic principal bundle $E$ over $X$.  The total space of $E$ is obtained from the disjoint union of $p^{-1}(U_i)\times G$, $i\in I$, by identifying $(x,g)\in p^{-1}(U_i)\times G$ with $(x,\psi_{ji}(x)g)\in p^{-1}(U_j)\times G$ when $x\in p^{-1}(U_{ij})$.  The action of $G$ on the fibres is by right multiplication.

We define a holomorphic $G$-action on $E$ by setting $h\cdot(x,g)=(hx, hg)$ for $(x,g)\in p^{-1}(U_i)\times G$ and $h\in G$.  The action is well defined since 
\[(hx,h\psi_{ji}(x)g) = (hx, \psi_{ji}(hx)hg)\]
is identified with $(hx,hg)$ by the $G$-equivariance of the cocycle.  The action commutes with the action of $G$ on the fibres (left and right multiplications commute), and the projection $E\to X$ is $G$-equivariant.  Thus, in the language of Lashof \cite{Lashof}, which we shall use below, $E$ is a holomorphic principal $G$-$G$-bundle over $X$, where the first $G$ acts on $E$ by a lift of its action on $X$, and the second $G$ acts on each fibre of $E$.  (In the language of \cite{Heinzner-Kutzschebauch}, $E$ is a holomorphic $G$-principal $G$-bundle, where the first $G$ acts on each fibre and the second by a lift of the action on the base.)

Recall that $X$ and $Y$ are $G$-biholomorphic over $Q$ if and only if the cocycle $(\psi_{ij})$ splits (as a cocycle of $G$-equivariant holomorphic maps to $G$).  Equivalently, $E$ is trivial as a holomorphic principal $G$-$G$-bundle, meaning that it is isomorphic as a holomorphic $G$-$G$-bundle to the trivial principal $G$-$G$-bundle $X\times G$ with the action $(h,h')\cdot(x,g)=(h x, h g h')$.  By the equivariant Oka-Grauert principle of Heinzner and Kutzschebauch \cite[p.~341]{Heinzner-Kutzschebauch}, this holds if (and of course only if) there is a topological $K$-equivariant isomorphism from $E$ to the trivial principal $G$-$G$-bundle, that is, if $E$ is trivial as a topological principal $K$-$G$-bundle.  Here, $K$ denotes a maximal compact subgroup of $G$.  (In our situation, the complexification $X^\C$ discussed in \cite{Heinzner-Kutzschebauch} is $X$ itself.)  

Equivalently, the cocycle $(\psi_{ij})$ splits as a cocycle of $K$-equivariant continuous maps to $G$.  In other words, there is a $K$-equivariant homeomorphism $\sigma:X\to Y$ of a certain form.  Namely, on $p^{-1}(U_i)$, we have $\sigma(x)=\phi_i(\gamma_i(x)x)$, where $\gamma_i:p^{-1}(U_i)\to G$ is continuous and $K$-equivariant.

In particular, we have proved the following result.

\begin{theorem}  \label{t:topological-obstruction}
Let $G$ be a reductive complex Lie group acting holomorphically and generically on normal Stein spaces $X$ and $Y$, which are locally $G$-biholomorphic over a common quotient.  The obstruction to $X$ and $Y$ being $G$-biholomorphic is topological.  Namely, there is a bundle naturally arising from the given data whose topological triviality is equivalent to $X$ and $Y$ being $G$-biholomorphic.
\end{theorem}

In the following sections, we will provide sufficient conditions for the obstruction to vanish, starting with the case when $G$ is abelian.

To close this section, we remark that to conclude that $E$ is trivial as a holomorphic principal $G$-$G$-bundle over $X$, it suffices to know that $E$ is trivial as a topological principal $K$-$G$-bundle over a Kempf-Ness set $R$ in $X$ \cite[p.~341]{Heinzner-Kutzschebauch}.  In other words, it suffices to split the cocycle $(\psi_{ij})$ on $R$.

We also remark, in the context of $G$ versus $K$, that if $X=\C^n$ or, more generally, $X$ has no nonconstant plurisubharmonic functions that are bounded above, then a holomorphic action of $G$ on $X$ is the same thing as an action (continuous or, equivalently, real analytic) of $K$ on $X$ by biholomorphisms \cite{Kutzschebauch}.  


\section{An example}  \label{s:example}

\noindent
Kraft and Schwarz have shown that if $G$ is a reductive complex Lie group and $X$ is an affine $G$-variety such the categorical quotient map $X\to X\sl G$ is flat (this is a stringent assumption), then the functor associating to a morphism $Z\to X\sl G$ the group of $G$-automorphisms of the pullback $Z\times_{X\sl G} X$ over $Z$ is represented by an affine group scheme over the quotient \cite[Prop.~III.2.2]{Kraft-Schwarz}.  The following example shows that this may fail when the quotient map is not flat, even for a low-dimensional representation of $\mathrm{SL}_2(\C)$.

\begin{example}  \label{e:canonical-bad-example}
We let $G=\mathrm{SL}_2(\C)$ and consider the $G$-module $V=\C^2\oplus\C^2\oplus\C^2\cong\C^6$ with the action $g\cdot(v_1,v_2,v_3)=(gv_1,gv_2,gv_3)$.  This is a well-studied action: see e.g.\ \cite{Weyl}, \cite{DeConcini-Procesi}, and \cite[Sec.~I.4]{Kraft1984}.  

The categorical quotient map is $\pi:V\to Q=\C^3$, $(v_1,v_2,v_3)\mapsto (f_3,f_2,f_1)$, where $f_1=\det [v_2\ v_3]$, $f_2=\det[v_1\ v_3]$, $f_3=\det[v_1\ v_2]$ \cite[Thm.~6.6]{DeConcini-Procesi}.  If $\pi(v_1,v_2,v_3)\neq 0$, that is, the three vectors span $\C^2$, then the $\pi$-fibre through $(v_1,v_2,v_3)$ is simply the orbit through $(v_1,v_2,v_3)$, and is isomorphic to $G$ with trivial stabiliser.  In particular, the action is generic with principal stratum $Q^*=Q\setminus\{0\}$.  The null cone $N=\pi^{-1}(0,0,0)$, which consists of triples of vectors $(v_1,v_2,v_3)$ that span a line or are all zero, is a vector bundle of rank $2$ over $\mathbb P^2$ with the zero section blown down to a point \cite[p.~28]{Kraft1984}.  The point corresponds to the triple $(0,0,0)$ and is the unique closed orbit in $N$.  The non-closed orbits in $N$ are the fibres of the vector bundle with zero removed. 

It is clear that the group of $G$-automorphisms of each principal fibre is $G$ itself.  In fact, over the principal stratum $Q^*$, since $G$ acts freely, $\pi$ is a principal $G$-bundle (\cite[Cor.~5]{Luna}, \cite[Cor.~5.5]{Snow}).  Thus, over $Q^*$, the sheaf $\A$ of $G$-biholomorphisms of $V$ over $Q$ is the sheaf of sections of a holomorphic principal $G$-bundle.

There is another useful action on $V$.  Viewing $V$ as $\C^2\otimes\C^3$, we see that there is an action of $H=\mathrm{GL}_3(\C)$ on $V$, commuting with the $G$-action.  Since the actions commute, the $H$-action descends to $Q$, and $\pi:V\to Q$ is $H$-equivariant.  Also, $H$ acts by conjugation on the $G$-automorphisms of $V$ over $Q$.  Clearly, $H$ preserves the null cone $N$, so each element of $H$ induces a $G$-automorphism of $N$.  

It is easily seen that $H$ injects into $\Aut^G(N)$.  In fact, $\Aut^G(N)=H$ (this is not required for our arguments below).  Namely, the occurrences of the representation $\C^2$ in $\O_\mathrm{alg}(N)$ are spanned by the three copies of $\C^2$ in degree 1, and these copies of $\C^2$ generate $\O_\mathrm{alg}(N)$.  Thus an element of $\Aut^G(N)$ permutes the copies of $\C^2$ linearly and corresponds to an element of $H$.

We will show that the sheaf $\A$ is not representable, in the sense that there is no group object $\alpha:A\to Q$ in the category of complex spaces over $Q$ representing the functor that takes a holomorphic map $f:Y\to Q$ of reduced complex spaces to the group of $G$-automorphisms of $Y\times_Q V$ over $Y$, meaning that there is a natural group isomorphism between the group $\Aut_Y^G(Y\times_Q V)$ and the group $\Hom(f,\alpha)$ of all holomorphic maps $g:Y\to A$ with $\alpha\circ g=f$, which is naturally identified with the group of holomorphic sections of $Y\times_Q A$ over $Y$.  Informally speaking, $\A$ has a bad singularity over $0$.  Since $Y$ is reduced, $\Hom(f,\alpha)$ is naturally identified with $\Hom(f,\tilde\alpha)$, where $\tilde\alpha:\tilde A\to Q$ is the reduction of $\alpha$.  Thus we may assume that $A$ is reduced.

Suppose such a representing $\alpha$ exists.  We will derive a contradiction.  First, letting $Y$ run through the points of $Q$, we see that the fibres of $\alpha$ over $Q^*$ are $G$.  The fibre over $0$ is $\Aut^G(N)\supset H$.  Over $Q^*$, $A$ is connected (since $G$ is connected) and 6-dimensional, so $\alpha^{-1}(Q^*)$ lies in a 6-dimensional irreducible component $C$ of $A$.  Clearly, $H\cap C\neq C$, so $H\cap C$ is a closed subgroup of $H$ of dimension at most 5.

We want to know that $H\cap C$ is a \textit{normal} subgroup of $H$, because this narrows it down drastically.  As a consequence of the Lie algebra $\mathfrak{sl}_3$ being simple, a normal subgroup of $H$ either consists of scalar matrices or contains $\mathrm{SL}_3(\C)$.  If $H\cap C$ is normal, since it is at most 5-dimensional, it must consist of scalar matrices, which is easily contradicted.  Indeed, take $f$ in the universal property to be the inclusion of the line $Y=\{(t,0,0):t\in\C\}$ into $Q$.  For $t\neq 0$,
\[\pi^{-1}(t,0,0)=\{(v_1, v_2, 0)\in V : \det[v_1\ v_2]=t\} \cong G, \]
and of course $\pi^{-1}(0,0,0)=N$.  These are the fibres of $Y\times_Q V\to Y$.  There is a $G$-automorphism of $Y\times_Q V$ over $Y$ given by $(v_1,v_2,v_3)\mapsto (2v_1, \tfrac 1 2 v_2, v_3)$.  Its restriction to $N$ is given by a non-scalar element of $H$.  (Alternatively, if we know that $\Aut^G(N)=H$, we can simply observe that the fibre dimension of $\alpha\vert C$ cannot drop from 3 down to 0 or 1 over $0\in Q$.)

To see that $H\cap C$ is normal, we verify that the action of $H$ by conjugation is built into $A$ via its universal property.  Let $h\in H$ and take $f$ in the universal property above to be $A \xrightarrow{\alpha} Q \xrightarrow{h} Q$.  Then we have a natural bijection between $\Hom(f,\alpha)$, which is the set of liftings of $h$ to $A$, and $\Aut_A^G(A\times_Q V)$. The desired action of $h$ on $A$ is the lifting corresponding to the $G$-automorphism 
\[ (\phi,v)\mapsto(\phi, (h\phi h^{-1})(v)) \]
of $A\times_Q V$ over $A$.
\end{example}


\section{Abelian reductive groups}  \label{s:torus}

\noindent
We use the notation established at the beginning of Section \ref{s:generic}.  In this section, we take $G$ to be abelian.  Then the conjugation action of $G$ on itself is trivial, so a $G$-equivariant map $p^{-1}(U)\to G$, where $U\subset Q$ is open, is simply a $G$-invariant map.  Thus, by the universal property of the categorical quotient, $\A$ may be identified with the sheaf of holomorphic maps from open subsets of $Q$ into $G$.

If $G$ is a torus $(\C^*)^k$, $k\geq 1$, then
\[ H^1(Q,\A)\cong H^1(Q,\O^*)^k\cong H^2(Q,\Z)^k \]
since $Q$ is Stein.  The second isomorphism is Oka's original Oka principle!  The following proposition shows that $H^2(Q,\Z)$ vanishes if $X$ is smooth and $\Z$-acyclic.

\begin{proposition}  \label{p:acyclic}
Let $G$ be a reductive complex Lie group acting holomorphically on a $\mathbb Z$-acyclic Stein manifold
 $X$.  Then $X\sl G$ is $\mathbb Z$-acyclic.
\end{proposition}

\begin{proof}
First, $X$ has a real-analytic $K$-invariant strictly plurisubharmonic exhaustion, and the corresponding Kempf-Ness set $R$ is a $K$-equivariant strong deformation retract of $X$ \cite[p.~23]{Heinzner-Huckleberry}.  Hence, $R/K$ is a strong deformation retract of $X/K$.  Also, $X\sl K=X\sl G$ is homeomorphic to $R/K$  \cite[p.~22]{Heinzner-Huckleberry}.  Therefore it suffices to show that the orbit space $X/K$ is $\mathbb Z$-acyclic.

By a theorem of Oliver \cite{Oliver}, to conclude that $X/K$ is $\mathbb Z$-acyclic, we need to know that $X$ is paracompact of finite cohomological dimension and with finitely many $K$-orbit types (finitely many conjugacy classes of stabilisers).  The first two conditions are evident.

To verify the third condition, we use a theorem of Mann \cite{Mann}, which states that a compact Lie group acting on an orientable cohomology manifold over $\mathbb Z$ with finitely generated integral cohomology has only finitely many orbit types.  We conclude that the action of $K$ on $X$ has only finitely many orbit types.
\end{proof}

The following theorem is now immediate for tori, and with a little more work we can prove it for abelian groups in general.

\begin{theorem}  \label{t:torus}
Let $X$ and $Y$ be Stein manifolds on which a reductive complex Lie group $G$ acts holomorphically and generically, such that $X$ and $Y$ are locally $G$-biholomorphic over a common quotient.  If $G$ is abelian and $X$ is $\mathbb Z$-acyclic, then $X$ and $Y$ are $G$-biholomorphic.
\end{theorem}

\begin{proof}
Since $G$ is abelian, it is an extension of a torus $T=(\C^*)^k$ by a finite abelian group $F$.  The short exact sequence $0\to T\to G\to F\to 0$ induces a short exact sequence
\[ 0 \to \O(\cdot,T) \to \O(\cdot,G)\to \O(\cdot,F)\to 0 \]
of sheaves of abelian groups on $Q=X\sl G$.  Clearly, $\O(\cdot, T)=(\O^*)^k$, $\O(\cdot, G)=\A$, and $\O(\cdot,F)$ is simply the sheaf of locally constant functions with values in $F$.  Consider the long exact sequence
\[ \cdots \to H^1(Q, \O^*)^k \to H^1(Q, \A) \to H^1(Q, F) \to \cdots.  \]
By Proposition \ref{p:acyclic}, $H^1(Q,\O^*)\cong H^2(Q,\Z)=0$.  By Proposition \ref{p:acyclic} and universal coefficients, $H^1(Q,F)=0$ (for this we require both $H^1(Q,\Z)$ and $H^2(Q,\Z)$ to vanish).  Hence $H^1(Q,\A)=0$, so the obstruction to $X$ and $Y$ being $G$-biholomorphic over $Q$ vanishes.
\end{proof}

Although we do not need it, we remark that arguments of Kraft, Petrie, and Randall \cite{Kraft-Petrie-Randall} in the algebraic case carry over to the analytic case and, combined with Proposition \ref{p:acyclic}, show that if $G$ is a reductive complex Lie group acting holomorphically on a contractible Stein manifold $X$, then $X\sl G$ is contractible.  If there is a connected $G$-orbit, for instance if $G$ is connected or has a fixed point in $X$, then the result follows easily from Proposition \ref{p:acyclic} and \cite[Cor.~II.6.3]{Bredon}.

A deep theorem of Hochster and Roberts \cite{Hochster-Roberts} states that the categorical quotient of a smooth affine variety by the action of a reductive group is Cohen-Macaulay.  Slice theory allows us to easily extend the theorem to the holomorphic setting.  We note that Cohen-Macaulay in the algebraic sense is equivalent to Cohen-Macaulay in the holomorphic sense.  Namely, by GAGA \cite[Prop.~3]{Serre}, each stalk of the algebraic structure sheaf has the same completion as the corresponding stalk of the holomorphic structure sheaf, so one is Cohen-Macaulay if and only if the other one is \cite[Prop.~18.8]{Eisenbud}.

\begin{proposition}  \label{p:cohen-macaulay}
Let a reductive complex Lie group $G$ act holomorphically on a Stein manifold $X$.  The categorical quotient $X\sl G$ is Cohen-Macaulay.
\end{proposition}

\begin{proof}
Let $q$ be a point in $X\sl G$, and let $x$ be a point in the closed orbit over $q$ with stabiliser $H$. By the slice theorem, a neighbourhood of $q$ in $X\sl G$ is biholomorphic to an open subset of the quotient of $W$ by the reductive group $H$, where $W$ is the slice representation of $H$.  The action of $H$ on $W$ is linear algebraic.  By \cite{Hochster-Roberts}, $W\sl H$ is Cohen-Macaulay, so the stalk $\O_{X\sl G,q}$ is Cohen-Macaulay.  Thus $X\sl G$ is Cohen-Macaulay
\end{proof}

\begin{theorem}  \label{t:codim-3}
Let $X$ and $Y$ be Stein manifolds on which a reductive complex Lie group $G$ acts holomorphically and generically, such that $X$ and $Y$ are locally $G$-biholomorphic over a common quotient $Q$.  Assume moreover that $G$ is abelian and that the complement of the principal stratum $Q_\mathrm{pr}$ has codimension at least 3 in $Q$.  If there is an equivariant homeomorphism $X_\mathrm{pr}\to Y_\mathrm{pr}$ over $Q_\mathrm{pr}$, then $X$ and $Y$ are $G$-biholomorphic over $Q$.
\end{theorem}

\begin{proof}
We use the same notation as in the proof of Theorem \ref{t:torus}.  Consider the commuting diagram
\[ \xymatrix{
H^0(Q_\mathrm{pr}, F) \ar[r] \ar@{=}[d] & H^1(Q_\mathrm{pr}, \O^*)^k \ar[r] \ar[d]^\alpha & H^1(Q_\mathrm{pr}, \A) \ar[r] \ar[d]^\beta & H^1(Q_\mathrm{pr}, F) \ar@{=}[d] \\
H^0(Q_\mathrm{pr}, F) \ar[r] & H^1(Q_\mathrm{pr}, \mathscr C^*)^k \ar[r] & H^1(Q_\mathrm{pr}, \mathscr C(\cdot, G)) \ar[r] & H^1(Q_\mathrm{pr}, F) } \]
with exact rows (here, $\mathscr C$ refers to continuous functions).  Let $\omega\in H^1(Q_\mathrm{pr}, \A)$ be the obstruction to $X_\mathrm{pr}$ and $Y_\mathrm{pr}$ being $G$-biholomorphic over $Q_\mathrm{pr}$.  By assumption, $\beta(\omega)=0$.  We will show that $\alpha$ is injective.  It follows that $\beta$ is injective, so $\omega=0$.  Since $X\setminus X_\mathrm{pr}$ has codimension at least 2 in $X$ by the genericity assumption, any $G$-biholomorphism $X_\mathrm{pr}\to Y_\mathrm{pr}$ over $Q_\mathrm{pr}$ extends to a $G$-biholomorphism $X\to Y$ over $Q$.

Now consider the commuting diagram
\[ \xymatrix{
H^1(Q_\mathrm{pr}, \Z) \ar[r] \ar@{=}[d] & H^1(Q_\mathrm{pr}, \O) \ar[r] \ar[d] & H^1(Q_\mathrm{pr}, \O^*) \ar[r] \ar[d] & H^2(Q_\mathrm{pr}, \Z) \ar@{=}[d] \\
H^1(Q_\mathrm{pr}, \Z) \ar[r] & H^1(Q_\mathrm{pr}, \mathscr C)=0 \ar[r] & H^1(Q_\mathrm{pr}, \mathscr C^*) \ar[r] & H^2(Q_\mathrm{pr}, \Z)  } \]
with exact rows.  By Proposition \ref{p:cohen-macaulay}, $Q$ is Cohen-Macaulay, and since $Q\setminus Q_\mathrm{pr}$ has codimension at least 3, the vanishing theorem for local cohomology of Scheja and Trautmann (\cite{Scheja}, \cite{Trautmann}, \cite[Thm.~1.14]{Siu-Trautmann}, \cite[Thm.~II.3.6]{Banica-Stanasila}) implies that $H^1(Q_\mathrm{pr},\O)\cong H^1(Q,\O)=0$.  Hence $\alpha$ is injective.
\end{proof}


\section{General reductive groups}

\noindent
We continue to use the notation established at the beginning of Section \ref{s:generic}.  We refer to Lashof's foundational paper for the results we need about the classification of topological $K$-$G$-bundles over $X$ (see also \cite{tomDieck}).  Since $G$ is isomorphic to a closed subgroup of $\mathrm{GL}_n(\C)$ for some $n$, every $K$-$G$-bundle over $X$ is numerable \cite[Prop.~1.11 and Cor.~1.13]{Lashof}.  (Numerability is a technical condition whose definition we omit.)

We say that $X$ is $K$-contractible if the identity map of $X$ can be joined to a constant map by a continuous path of $K$-equivariant continuous maps $X\to X$.  The value of the constant map must then be a fixed point $x_0$ of the $K$-action.  A $K$-module is obviously $K$-contractible.

By \cite[Cor.~2.11]{Lashof}, if $X$ is $K$-contractible, then every numerable $K$-$G$-bundle over $X$ is isomorphic as a $K$-$G$-bundle to its own pullback by the map that takes all of $X$ to $x_0$.  Thus, by the discussion preceding Theorem \ref{t:topological-obstruction}, we have the following theorem and corollaries.

\begin{theorem}   \label{t:K-contractible}
Let $G$ be a reductive complex Lie group with maximal compact subgroup $K$.  Let $X$ and $Y$ be normal Stein spaces on which $G$ acts holomorphically and generically, such that $X$ and $Y$ are locally $G$-biholomorphic over a common quotient.  If $X$ is $K$-contractible, then $X$ and $Y$ are $G$-biholomorphic.
\end{theorem}

\begin{corollary}  \label{c:first-corollary}
Let $X$ be a Stein manifold  on which $G$ acts holomorphically and let $V$ be a $G$-module, such that $X$ and $V$ are locally $G$-biholomorphic over a common quotient.  If $X$ and $V$ are generic, then $X$ and $V$ are $G$-biholomorphic.
\end{corollary}

\begin{corollary}  \label{c:Franks-dream}
A holomorphic $G$-action on $\C^n$, which is locally $G$-biholomorphic over a common quotient to a generic linear action, is linearisable.
\end{corollary}


\section{Strict equivariant diffeomorphisms}  \label{s:smooth}

\noindent
Let $X$ and $Y$ be Stein manifolds on which a reductive complex Lie group $G$ acts holomorphically, such that $X$ and $Y$ are locally $G$-biholomorphic over a common quotient $Q$ with quotient maps $p:X\to Q$ and $r:Y\to Q$.  A $G$-equivariant diffeomorphism $X\to Y$ which induces the identity on $Q$ and is biholomorphic from $p^{-1}(q)$ onto $r^{-1}(q)$ for every $q\in Q$ will be called a \textit{strict $G$-diffeomorphism} from $X$ to $Y$.  If a strict $G$-diffeomorphism $X\to Y$ exists, then we call $X$ and $Y$ \textit{strictly $G$-diffeomorphic}.

A $G$-equivariant diffeomorphism $\psi$ of $X$ is called \textit{special} if it is of the form $\psi (x) = \gamma (x) \cdot x$ for some smooth $G$-equivariant map $\gamma: X \to G$, where $G$ acts on the target $G$ by conjugation.  A $G$-equivariant diffeomorphism $\psi:X\to Y$ is called \textit{special} if for some open cover $(U_i)$ of $Q$ and $G$-biholomorphisms $\phi_i:p^{-1}(U_i)\to r^{-1}(U_i)$ over $U_i$, the diffeomorphisms $\phi_i^{-1} \circ \psi$ of $p^{-1}(U_i)$ are special.

If the action on $X$ (and thus on $Y$) is generic, then every $G$-biholomorphism over the quotient is  a special (and obviously a strict) $G$-diffeomorphism (Lemma \ref{l:g-biholomorphisms}).  Over the principal stratum of a generic action, the notions of a special and a strict $G$-diffeomorphism coincide.  Clearly, the special $G$-diffeomorphisms of $X$ form a group, and so do the strict ones.

If the action is generic, then the definition of a special $G$-diffeomorphism $X\to Y$ does not depend on the choice of the cover $(U_i)$ of $Q$ and the $G$-biholomorphisms $\phi_i:p^{-1}(U_i)\to r^{-1}(U_i)$ over $U_i$, since any $G$-biholomorphism over the quotient is special, and the special $G$-diffeomorphisms form a group.

\begin{theorem}  \label{t:strict-diffeomorphisms}
Let $G$ be a reductive complex Lie group acting holomorphically and generically on Stein manifolds $X$ and $Y$, which are locally $G$-biholomorphic over a common quotient.  Every strict $G$-diffeomorphism $X\to Y$ is homotopic, via a continuous path of strict $G$-diffeomorphisms, to a special (and strict) $G$-diffeomorphism.
\end{theorem}

By the discussion preceding Theorem \ref{t:topological-obstruction}, the obstruction to $X$ and $Y$ being $G$-biholomorphic over $Q$ vanishes if there is a special $G$-diffeomorphism $X\to Y$.  Hence the following corollary is immediate.

\begin{corollary}  \label{c:strict-diffeomorphisms}
Let $G$ be a reductive complex Lie group acting holomorphically and generically on Stein manifolds $X$ and $Y$, which are locally $G$-biholomorphic over a common quotient.  If $X$ and $Y$ are strictly $G$-diffeomorphic, then they are $G$-biholomorphic.
\end{corollary}

The remainder of this section is devoted to the proof of the theorem.  We start by constructing the desired homotopy in a particular local setting.  Let $H$ be a reductive subgroup of $G$, and let $W$ be an $H$-module, not necessarily generic, such that $W^H=0$.  Let  $p_1,\ldots,p_k$ be homogeneous generators of $\O_\mathrm{alg}(W)^H$ of degrees $d_1, \ldots, d_k$ respectively, and let $B=\{w\in W : \lvert p_i\rvert<a_i\}$ for some $a_1,\ldots,a_k>0$.  Let $T_W$ (resp.\ $T_B$) be the tube $G\times^HW$ (resp.\ $G\times^HB$).  Then $T_W$ and $T_B$ are bundles over $G/H$ with fibres $W$ and $B$, respectively.  Below, when talking about derivatives in the fibre-direction, we mean the fibres of these bundles.  Note that the $G$-action on $T_W$ is generic if and only if the $H$-action on $W$ is generic.

The null fibre of the quotient map $T_B \to T_B\sl G\cong B\sl H$ is $G\times^H\N(W)$, where $\N (W)$ denotes the null cone of the $H$-representation on $W$.  The unique closed orbit in the null fibre is the zero section $Z$ of $T_W$.  Now let $\phi:T_B\to T_B$ be a strict $G$-diffeomorphism.  Then $\phi$ must preserve $Z$.  Let $\delta\phi:T_W\to T_W$ denote the derivative of $\phi$ in the fibre-direction along $Z$.

For $t\in \C^*$, denote by $\alpha_t : T_W\to T_W$  the $G$-biholomorphism   defined by $\alpha_t ([g, w])= [g,tw]$.  Note that $p_1,\ldots,p_k$ correspond to generators $F_1,\ldots,F_k$ of $\O_\mathrm{alg}(T_W)^G$ with $F_i ([g, w])= p_i (w)$, where $F_i \circ \alpha_t = t^{d_i} F_i$.  Let $\phi_t = \alpha_t^{-1} \circ \phi \circ \alpha_t$ for $ t>0$, and let $\phi_0=\delta\phi$.  

\begin{lemma} 
The family $\phi_t$, $t\in[0,1]$, is a homotopy of strict $G$-diffeomorphisms of $T_B$ joining $\phi$ to $\delta \phi$.   If the $H$-action on $W$ is generic, then $\phi_0$ is a special (and strict) $G$-diffeomorphism.
\end{lemma}

\begin{proof}
Since $\alpha_t$ is holomorphic, the maps $\phi_t : T_W\to T_W$,   $t\in(0,1]$, are strict $G$-diffeomorphisms.  Moreover, they induce the identity map on the quotient since they preserve the invariants:
\begin{align*}
F_i ( \alpha_t^{-1} \circ \phi \circ \alpha_t ([g, w])) &= t^{-d_i} F_i (\phi \circ  \alpha_t ([g, w])) =  t^{-d_i} F_i ( \alpha_t ([g, w])) \\ &=  t^{-d_i} t^{d_i}F_i ([g, w]) =  F_i ([g, w]).
\end{align*}
This shows, in particular, that for every $t\in (0,1]$, $\phi_t$ is a strict $G$-diffeomorphism of $T_B$ and that $\phi_0$ preserves the invariants. 
By differentiability of $\phi$,   $\lim\limits_{t\to 0} \phi_t$ exists locally uniformly on $T_B$ and equals $\phi_0=\delta\phi $.  Therefore, $\phi_0$ is a $G$-diffeomorphism over the quotient. 

By Lemma \ref{l:g-biholomorphisms}, to prove that $\phi_0$ is special in the generic case, it suffices to show that $\phi_0$ is holomorphic, that is, that the the derivative of $\phi$ in the fibre-direction along $Z$ is complex-linear and not merely real-linear.  By assumption, $\phi$ restricts to a $G$-biholomorphism of $G\times^H \N(W)$, so it has a complex-linear derivative in the fibre-direction along the Zariski tangent space of $\N(W)$.  Since $W^H=0$ by assumption, the Zariski tangent space at $0\in W$ to $\N(W)$ is all of $W$.  This shows that $\phi_0$ is complex-linear.
\end{proof}

We can explicitly describe the form of $\phi_0$.  We have $\phi_0([g,w])=\gamma([g,w])\cdot ([g,w])$ with $\gamma:G\times^HW\to G$ equivariant and algebraic.  Then $\phi_0([e,hw])=h\delta\phi[(e,w])h\inv$, so that $\gamma([e,w]):W\to G$ is $H$-equivariant.  But for $t\in\R$, $\phi_0([e,tw])=t\delta\phi([e,w])$, so 
\[ t[\gamma([e,tw]),w]=[\gamma([e,tw]),tw]=t[\gamma([e,w]),w]. \]
Hence $\gamma([e,w])\cdot [e,w]=\gamma([e,0])\cdot [e,w]$, where $\gamma([e,0])$ is an element $g_0$ of $G$ centralising $H$.  Thus $\gamma([g,w])([g,w])=([gg_0,w])$ and $\gamma([g,w])=gg_0g\inv$.  The element $g_0$ is unique, since for $w$ a principal point of $W$, $\gamma([e,w])$ is unique (and equal to $g_0$).  It is easy to see that any choice of $g_0$ in the centraliser of $H$ in $G$ gives an equivariant $\gamma:G\times^H W \to G$. 

Next we describe the homotopy in the general local setting.  In addition to the notation above, let $C$ be a Stein manifold (with a trivial $G$-action), and let $ \phi:C \times T_B\to C\times T_B$ be a strict $G$-diffeomorphism.  Since $\phi$ induces the identity on the quotient, it is of the form $\phi (c, [g, w]) = (c, \tilde \phi (c, [g,w]))$, where $\tilde \phi : C \times T_B \to T_B$ may be viewed as a smooth family of strict $G$-diffeomorphisms of $T_B$ parametrised by $C$.  Let $\delta \tilde  \phi : C \times T_W\to T_W$ be the derivative of $\tilde \phi$ in the fibre-direction along the zero section of $T_W$, and set $\delta\phi (c, [g,w]) = (c, \delta \tilde \phi (c, [g,w]))$.

Applying the above homotopy in this parametrised setting, we get the following result.  We let $\alpha_t$ act on $C\times T_B$ as the identity on the first factor and as defined above on the second factor.

\begin{lemma}  \label{l:parameters}
The family $\phi_t : C \times T_B \to C\times T_B$, $t\in[0,1]$, with 
\[ \phi_t (c, [g, w])=(c, \alpha_t^{-1} \circ \tilde \phi \circ \alpha_t (c, [g,w]))\] 
for $t>0$ and 
\[\phi_0 (c, [g, w])=(c, \delta \tilde \phi (c, [g, w])),\] 
is a homotopy of strict $G$-diffeomorphisms joining $\phi$ to $\delta \phi $.  If the $G$-action on $T_B$ (equivalently, on $C \times T_B$) is generic, then $\phi_0$ is a special (and strict) $G$-diffeomorphism.
\end{lemma}

We now want to change our homotopy so that it is the identity away from a neighbourhood  of a given point.  Let $c_0\in C$ and let $\rho: C\times T_B\sl G\to[0,1]$ be a smooth function which is 1 on a neighbourhood $U$ of $p(\{c_0\}\times Z)$ and has compact support $F$.  Let $\tau(t,z)=1+(t-1)\rho(z)$ for $t\in[0,1]$ and $z\in C\times T_B\sl G$.  Then $\tau(t,z)=1$ for $z\notin F$, and $\tau(t,z)=t$ for $z\in U$.  Let $\tau(t,z)\cdot (c,[g,w])$ denote $(c,[g,\tau(t,z)w])$ for $(c,[g,w])\in C\times T_B$, where $z=p(c,[g,w])$.

\begin{corollary}  \label{c:phi'}
Let $\rho$ be as above and let $\phi:C\times T_B \to C\times T_B$ be a strict $G$-diffeo\-morphism.  Let $\phi_t(x)=\tau(t,z)\inv\cdot \phi(\tau(t,z)\cdot x)$ for $x\in C\times T_B$, $z=p(x)$, and $t\in(0,1]$.  Set $\phi_0=\lim\limits_{t\to 0}\phi_t$.  The family $\phi_t$, $t\in[0,1]$, is a homotopy of strict $G$-diffeomorphisms joining $\phi$ to $\phi_0$.  Moreover, for each $t\in[0,1]$, $\phi_t$ equals $\phi$ over the complement of $F$, and $\phi_0$ equals $\delta\phi$ over $U$.
\end{corollary}

\begin{lemma}  \label{l:special-remains} 
If in the above situation there is a $G$-saturated open subset $\Omega$ of $C\times T_B$ which is  invariant under $x\mapsto \tau(t,z)\cdot x$   such that the restriction of $\phi$ to $\Omega$ is special, then all the strict $G$-diffeomorphisms $\phi_t$ are special when restricted to $\Omega$.
\end{lemma}

\begin{proof}
The assumption is that for $(c, [g, w]) \in \Omega$, we have $\phi (c, [g, w])  = (c, \gamma(c, [g, w]) \cdot [g, w])$, where $\gamma : \Omega \to G$ is smooth and $G$-equivariant.  It follows from the definition of $\phi_t$ that $\phi_t (c, [g, w])  = (c, \gamma_t (c, [g,  w])\cdot  [g, w])$, where $ \gamma_t (x) = \gamma (\tau(t,z)\cdot x)$, which is well defined by hypothesis.  The map $\gamma_t$ is clearly smooth and $G$-equivariant.
\end{proof}

\begin{proof}[Proof of Theorem \ref{t:strict-diffeomorphisms}] 
Let $\psi:X\to Y$ be a strict $G$-diffeomorphism.  Consider the stratification of $Q$ by the connected components of the Luna strata.  Let $S_k$ denote the union of the strata of dimension $k$.  We will inductively find homotopies of $\psi$ such that it becomes special over an open neighbourhood $\Omega$ of $S_0 \cup \cdots \cup S_k$.  Since over the principal stratum any strict $G$-diffeomorphism is special, we are done once we reach the case $k=\dim Q-1$.  Each step of the finite induction will be done by a  countable induction.   
 
Let $S$ be a stratum of minimal dimension.  Let $H$ be a corresponding stabiliser and let $(W,H)$ be the nontrivial part of the slice representation.  Cover $S$ by a locally finite collection of compact sets $K_i$, $i\in\mathbb N$, such that each $K_i$ lies in an open subset $U_i$ of $Q$ with $p\inv(U_i)\cong C_i\times (G\times^HB_i)$, where $C_i$ is a Stein open subset of $S$, and $B_i\subset W$ is as before.  We may assume that we have the same decomposition of $r\inv(U_i)$, so we may view $\psi$ over $U_i$ as a map of $C_i\times(G\times^HB_i)$ to itself.  We may also assume that any $K_j$ which intersects $K_i$ is contained in $U_i$.  By induction we suppose that there is a neighbourhood $\Omega_{n-1}$ of $K_1\cup\dots\cup K_{n-1}$ such that $\psi$ is strict and special on $p\inv(\Omega_{n-1})$.  Let $\rho_n(z)$ be smooth, $0\leq\rho_n\leq 1$, such that $\rho_n=1$ on a neighbourhood of $K_n$ and $\rho_n$ has compact support in $U_n$.  Then Corollary \ref{c:phi'} provides a homotopy from $\psi$ to $\psi'$, where $\psi'$ is special on a neighbourhood of $K_n$.  Now let $\beta_t(x)$ be the endomorphism of $U_n$ which is induced by the endomorphism of $p^{-1}(U_n)$ which sends $(c,[g,w])$ to $(c,[g,\tau(t,z)w])$, where $z$ is the image of $(c,[g,w])$ in $U_n$ and $\tau(t,z)=1+(t-1)\rho_n(z)$ as before.  Suppose that $K_j$, $j<n$, does not intersect $K_n$.  Then we may assume that $\rho_n$ vanishes on $(C_n\cap K_j)\times(G\times^HB_n)$, where we identify $C_n$ with its image in $S$.

Note that $\beta_t$ smoothly extends to be the identity outside of $U_n$. It  is also the identity on $U_{n}\cap S$ and on $K_1\cup\cdots\cup K_{n-1}$.  Thus there is a neighbourhood $\Omega_{n-1}'\subset\Omega_{n-1}$ of $K_1\cup\dots\cup K_{n-1}$ such that $\beta_t(z)$ maps $\Omega_{n-1}'$ into $\Omega_{n-1}$.  By Lemma \ref{l:special-remains}, $\phi'$ is special over $\Omega_{n-1}'\cup U_n'$, where $U_n'$ is a neighbourhood of $K_n$.  This gives us $\Omega_n$ and we continue inductively.  We could run into a problem if we are continually shrinking the neighbourhoods $\Omega_n$ near a point $z\in S$.  But there is a neighbourhood of $z$ which does not intersect the support of any $\rho_n$ for $n$ sufficiently large.  Thus for $n$ sufficiently large, each $\Omega_n$ contains a fixed neighbourhood of $z$.  Hence there is  a neighbourhood $\Omega$ of $S$ and a homotopy of $\phi$ to $\phi'$, such that $\phi'$ is strict and special over $\Omega$.  Since the strata of minimal dimension are closed and disjoint, we may find $\Omega$ as desired for all the strata of minimal dimension.

We may assume by induction that $\phi$ is special on a neighbourhood $\Omega$ of the union of the strata of dimension less than $m\geq 0$ (the union is a closed set).  Let $\Omega_1$ and $\Omega_2$ be smaller neighbourhoods  such that $\overline{\Omega}_1 \subset \Omega_2\subset \overline{\Omega}_2\subset\Omega$.  Let $S$ be a stratum of dimension $m$.  Choose a locally finite covering of $S$ by compact sets $K_i$ as above.  We may suppose that each $K_i$ either lies entirely inside $\Omega$,  lies entirely outside $\overline{\Omega}$, or does not intersect $\overline{\Omega}_2$.  We can now apply the above process to the $K_i$ that lie outside $\overline{\Omega}$, and to the $K_i$ that do not intersect $\overline{\Omega}_2$ to find a homotopy of $\psi$ such that it remains the same on $\Omega_1$ and becomes special on a neighbourhood of $S$.  We can do this for all the strata of dimension $m$, completing our induction.
\end{proof}


\section{Strong equivariant homeomorphisms}  \label{s:continuous}

\noindent
As before, we let $ G$ be a reductive complex Lie group, and $X$ and $Y$ be Stein manifolds on which $G$ acts holomorphically, such that $X$ and $Y$ are locally $G$-biholomorphic over a common quotient $Q$.  Let $p:X \to Q$ and $r:Y\to Q$ be the quotient maps.  In the previous section we showed that in the generic case we can deform a strict $G$-diffeomorphism $X\to Y$ to one which is also special.  In this section we define the notion of a \textit{strong $G$-homeomorphism} $X\to Y$, and show, in the generic case, that it can be deformed to a special $G$-homeomorphism, that is, one locally of the form $x\mapsto\gamma(x)\cdot x$, where $\gamma: X\to G$ is continuous and $G$-equivariant, with $G$ acting on the target $G$ by conjugation.  The proof is largely the same as in the previous section, aside from some technical lemmas.

Let $H$ be a reductive subgroup of $G$ and let $W$ be an $H$-module.  Let $B$ be an $H$-saturated neighbourhood of $0\in W$.  Let $T_W=G\times^HW$ and $T_B=G\times^HB$.  There is a $G$-stable finite-dimensional subspace $V\subset \O_\mathrm{alg}(T_W)$ which generates $\O_\mathrm{alg}(T_W)$, and also generates the $G$-finite elements of $\O(T_B)$ as a module over $\O(T_B)^G\cong \O(B)^H$.  We may think of the elements of $V$ as $G$-equivariant maps from $T_W$ to $V^*$ sending $[g,w]\in T_W$ to the element of $V^*$ whose value at $F\in V$ is $F([g,w])$.  Let $\{F_i\}$ be a basis of $V$.  We say that a $G$-equivariant homeomorphism $\psi:T_B\to T_B$ is \textit{strong} if it lies over the identity of $T_B\sl G$ and sends $F_i$ to $\sum_{j} a_{ij}F_j$, where the $a_{ij}$ are continuous functions on $T_B\sl G$.  It is easy to see that the definition does not depend on our choice of $V$ and the generators $F_i$.  Since the $F_i$ generate the coordinate ring of every fibre of $p:T_B\to Q_B=T_B\sl G$, $\psi$ restricts to an algebraic isomorphism of each fibre of $p$.  

Now let $\psi:X\to Y$ be a $G$-homeomorphism over $Q$.  Let $(U_i)$ be an open cover of $Q$ such that there are $G$-biholomorphisms $\phi_i: p\inv(U_i)\to r\inv(U_i)\cong G\times^{H_i}B_i$, where $B_i$ is an $H_i$-saturated neighbourhood of $0\in W_i$, $H_i$ is a reductive subgroup of $G$, and $W_i$ is an $H_i$-module.  We say that $\psi$ is \textit{strong} if each $\phi_i\inv\circ\psi:G\times^{H_i}B_i\to G\times^{H_i}B_i$ is strong.  Again, the definition does not depend on the choices made.

We may think of a strong $G$-homeomorphism $X\to X$ as a particular type of continuous family of algebraic (equivalently, holomorphic) $G$-isomorphisms of the fibres of $p$.  If the group scheme corresponding to $X$, as constructed in \cite{Kraft-Schwarz}, existed, then the strong $G$-homeomorphisms $X\to X$ would be the global continuous sections of the group scheme.  Note that a strong $G$-diffeomorphism is also strict.

Our main result in this section is the following counterpart to Theorem \ref{t:strict-diffeomorphisms} and Corollary \ref{c:strict-diffeomorphisms}.

\begin{theorem}  \label{t:strongly}
Let $G$ be a reductive complex Lie group acting holomorphically and generically on Stein manifolds $X$ and $Y$, which are locally $G$-biholomorphic over a common quotient.  Every strong $G$-homeomorphism $X\to Y$ is homotopic, via a continuous path of strong $G$-homeo\-morphisms, to a special (and strong) $G$-homeomorphism.
 
Hence, if $X$ and $Y$ are strongly $G$-homeomorphic, then they are $G$-biholomorphic.
\end{theorem}

The key to the theorem is Lemma \ref{l:fundamental} below which says that strongly continuous maps are somewhat differentiable.

Let $H$, $W$, $B$, and the $F_i$ be as above.  Assume that $W^H=0$.  Let $p_1,\ldots,p_m$ be homogeneous generators of $\O_\mathrm{alg}(W)^H$.  We have a grading on $\O_\mathrm{alg}(T_W)=\bigoplus\O_\mathrm{alg}(T_W)_r$ corresponding to the $\C^*$-action on $W$.  Thus $F_i$ has degree $r$ if $F_i([g,tw])=t^rF_i([g,w])$.  The elements of degree zero are the pullbacks to $\O_\mathrm{alg}(T_W)$ of the elements of $\O_\mathrm{alg}(G/H)$.  Now $\O_\mathrm{alg}(T_W)$ is the direct sum of covariants of the various degrees $r\geq 0$.  Let $d_j$ be the degree of $F_j$.  We ignore the covariants of degree zero, that is, we do not put them in our list.  We also ignore the invariants, that is, we only consider $F_j$ which transform by a nontrivial representation $V_i$ of $G$.  We may assume that $d_j=1$ if and only if $j\leq k$.   
 
\begin{lemma}  \label{l:gens}
$\O_\mathrm{alg}(T_W)$ is generated by $\O_\mathrm{alg}(T_W)_1$ as an $\O_\mathrm{alg}(G/H)$-algebra.
\end{lemma}
 
\begin{proof}
Let $W_g=\{[g,w]:w\in W\}$.  The restrictions of the $F_i$, $i\leq k$, must generate $W_g^*$,  since all the covariants $F_j$ restricted to $W_g$ generate $\O_\mathrm{alg}(W_g)$ and the covariants of higher degree obviously cannot help.  We are also using the assumption that $W^H=0$.  Thus the map $T_W\to G/H\times\C^k$ is an embedding, where the map to $\C^k$ comes from the $F_i$, $i\leq k$.  Thus $\O_\mathrm{alg}(T_W)$ is generated by $\O_\mathrm{alg}(T_W)_1$ over $\O_\mathrm{alg}(G/H)$.
\end{proof}
 
Let $\psi:T_B\to T_B$ be a strong $G$-homeomorphism.  We assume that $B$ is stable under multiplication by $t\in[0,1]$.  Let $(a_{ij}(z))$ be the $n\times n$ matrix of continuous functions such that $\psi^*F_i=\sum a_{ij}(z)F_j$, $z\in Q_B=T_B\sl G$.  Note that $\psi^*$ has to send the covariants corresponding to the $G$-module $V_i$ to the covariants of the same type.    
 
For $x\in T_B$, let $z$ denote $p(x)\in Q_B$, and let $x\mapsto t\cdot x$ denote the action of $t\in [0,1]$, that is, $t\cdot [g,w]=[g,tw]$.  Let $t\cdot z$ denote $p(t\cdot x)$.  Let $\psi_t(x)=t\inv\cdot \psi(t\cdot x)$, $x\in X$, $t\in[0,1]$.  As observed before, $\psi$ preserves the zero section $Z$ of $T_B$.  Thus $d\psi$ restricted to the Zariski tangent space at the origin of $\N(W_e)$ (which is $W_e$) gives us a complex-linear map of $W_e$ to $W_g$, where  $\psi$  sends $[e,0]$ to $[g,0]$.  We calculate this derivative in the usual way below.
 
\begin{lemma}  \label{l:fundamental}
In the above setting, the following hold.
\begin{enumerate}
\item  $(\psi_t^*F_i)(x)=\sum_j t^{d_j-d_i}a_{ij}(t\cdot z) F_j(x)$, $x\in T_B$. 
\item  The limit as $t\to 0$ of $\psi_t$ acts on the $F_i$, $i\leq k$, by the matrix $L$ with entries $a_{ij}(0)$, $i, j\leq k$.
\item  $\psi$ has a normal derivative $\delta\psi $ along the zero section of $T_B$, and $\delta\psi$  is a complex-linear $G$-vector bundle isomorphism of $T_W$ which preserves the invariants.
\item  There are continuous $b_{ij}(t,z)$ such that $(\psi_t^*F_i)(x)=\sum_j b_{ij}(t,z)F_j(x)$ for all $i$.
\item  If $W$ is generic, then $\psi_0=\delta\psi$ is special.
\end{enumerate}
\end{lemma}

\begin{proof} As seen before, the $\psi_t$ lie over the identity of $Q_B$.  Now 
\begin{align*}
(\psi_t^*F_i)(x)&=F_i(t\inv\psi(t\cdot x))=t^{-d_i}(\psi^*F_i)(t\cdot x)=t^{-d_i}\sum_j a_{ij}(t\cdot z)F_j(t\cdot x) \\ &=\sum_j t^{d_j-d_i}a_{ij}(t\cdot z)F_j(x) 
\end{align*}
and we have (1).  Now let $i\leq k$, so that $d_i=1$.  Then
\[  \lim_{t\to 0}\psi_t^*F_i(x)=\lim_{t\to 0}\sum_j t^{d_j-1}a_{ij}(t\cdot z) F_j(x).  \]
Since $d_j\geq 1$, and $d_j>1$ for $j>k$, the limit is $\sum\limits_{j\leq k}a_{ij}(0)F_j(x)$. Thus the directional derivative of $\psi$ in the direction of $x$ exists and the derivative acting on the $F_i$ for $i\leq k$ is given by the matrix $L$.  

Next we prove differentiability.  Let $\delta\psi$ be the bundle map of $T_B$ given by the directional derivatives.  Choose a norm on the vector bundle $T_W$ and consider $x\in T_B$ of norm at most $\epsilon$ for some $\epsilon>0$.  We must show that $\lim\limits_{t\to 0} t\inv(\psi(tx)-\delta\psi(tx))=0$ locally uniformly for such $x$.  For $i\leq k$, we have 
\begin{align*}
\lim_{t\to 0}F_i(t\inv\psi(tx)-\delta\psi(tx))&=\lim_{t\to 0}(\psi_t^*F_i-\delta\psi^*F_i)(x) \\ &=\lim_{t\to 0}\sum_j t^{d_j-1}(a_{ij}(t\cdot z)-a_{ij}(0)) F_j(x).
\end{align*}
The last limit vanishes since $a_{ij}$ is continuous at $z=0$, and the vanishing is locally uniform in $x$.  Hence $\psi$ is differentiable in the normal direction and $(\delta\psi)^*$ acts on those $F_i$ of degree 1 by the matrix $L$.  Thus $\delta\psi$ coincides with $d\psi$ calculated on the Zariski tangent spaces of the null cones; hence it is complex-linear.  Since $\psi$ is $G$-equivariant, so is $\delta\psi$, and since the $\psi_t$ preserve the invariants, so does $\delta\psi$.

In order to prove (4), we note that by Lemma \ref{l:gens}, for $i>k$, $F_i=\sum_s f_s Q_s(F_1,\dots,F_k)$, where each $Q_s$ is homogeneous of degree $d_i$ and $f_s\in\O_\mathrm{alg}(G/H)$.  When we expand any $Q_s$, we get an  expression in the invariants times the generators $F_j$. If $F_j$ occurs, then since we have an expression starting in degree at least $d_i$, either $d_j>d_i$, or the coefficient of $F_j$ has to be an invariant with lowest-degree term of degree at least $d_i-d_j$.

Expanding  $\psi_t^*F_i$ using the   expression above for $F_i$, we obtain
\begin{align*}
(\psi_t^*F_i)(x)&=\sum_s f_sQ_s(\psi_t^*F_1,\dots,\psi_t^*F_k)(x) \\ &=\sum_s t^{-d_i}f_s Q_s\left(\sum_j a_{1j}(t\cdot z)F_j,\dots,\sum_j a_{kj}(t\cdot z) F_j\right)(t\cdot x).
\end{align*}
Hence 
\[  (\psi_t^*F_i)(x)=\sum_s\sum_jf_{sj} t^{d_j-d_i}a'_{ijs}(t\cdot z) P_{ijs}(t\cdot x) F_j(x),  \] 
where $a'_{ijs}$ is a polynomial in the $a_{pq}$, $P_{ijs}$ is an invariant, and $f_{sj}\in\O_\mathrm{alg}(G/H)$.  If $d_j\geq d_i$, then the coefficient of $F_j(x)$ has the desired form.  If $d_j<d_i$, then the lowest-degree monomial of $P_{ijs}$ has degree at least $d_i-d_j$.  Set $P'_{ijs}(t,z)=t^{d_i-d_j}P_{ijs}(t\cdot x)$. Then $P'_{ijs}$ is continuous on $\R\times Z$ and the coefficient of $F_j(x)$ in $(\psi_t^*F_i)(x)$ is $\sum_s f_{sj}a'_{ijs}(t\cdot z)P'_{ijs}(t, z)$ giving (4). Part (5) is clear.
\end{proof}

\begin{proof}[Proof of Theorem \ref{t:strongly}] 
Using the arguments of the previous section, we prove the analogue of Lemma  \ref{l:parameters} for continuous parameters.  Using (4) above, we establish the analogue of Corollary \ref{c:phi'}, where the homotopy is through strong $G$-homeomorphisms, and the continuous analogue of Lemma \ref{l:special-remains} is straightforward.  Now we can simply repeat the proof of Theorem \ref{t:strict-diffeomorphisms} using the analogues of the lemmas and corollary.
\end{proof}


\end{document}